\documentclass[11pt]{amsart}

\usepackage{amssymb}
\usepackage{amsmath}
\usepackage{amsfonts}
\usepackage{graphicx}
\usepackage{amsthm}
\usepackage{enumerate}
\usepackage[mathscr]{eucal}
\usepackage{mathrsfs}
\usepackage{verbatim}
\usepackage{yhmath}

\usepackage{epstopdf}

\usepackage{color}

\makeatletter
\@namedef{subjclassname@2010}{%
  \textup{2010} Mathematics Subject Classification}
\makeatother

\numberwithin{equation}{section}
\numberwithin{figure}{section}

\theoremstyle{plain}
\newtheorem{theorem}{Theorem}[section]
\newtheorem{lemma}[theorem]{Lemma}
\newtheorem{proposition}[theorem]{Proposition}

\theoremstyle{plain}

\theoremstyle{remark}
\newtheorem{remark}[theorem]{Remark}

\DeclareMathOperator{\vol}{vol}

\begin{document}
\date{}

\title
[Eigenvalues and lattice points]{A note on the Weyl formula for balls in $\mathbb{R}^d$}

\author{Jingwei Guo}
\address{School of Mathematical Sciences\\
University of Science and Technology of China\\
Hefei, 230026\\ P.R. China\\}
\email{jwguo@ustc.edu.cn}

\thanks{J.G. is partially supported by the NSFC Grant (No. 11501535 and 11571331) and the Fundamental Research Funds for the
Central Universities (No. WK3470000013). }



\begin{abstract}
Let $\mathscr{B}=\{x\in\mathbb{R}^d : |x|<R \}$ ($d\geq 3$) be a ball. We consider the Dirichlet Laplacian associated with $\mathscr{B}$ and prove that its eigenvalue counting function has an asymptotics
\begin{equation*}
\mathscr{N}_\mathscr{B}(\mu)=C_d \vol(\mathscr{B})\mu^d-C'_d\vol(\partial \mathscr{B})\mu^{d-1}+O\left(\mu^{d-2+\frac{131}{208}}(\log \mu)^{\frac{18627}{8320}}\right)
\end{equation*}
as $\mu\rightarrow \infty$.
\end{abstract}

\maketitle


\section{Introduction}\label{sec1}

Let $D \subset \mathbb R^d$ ($d\geq 2$) be an open bounded connected domain with piecewise smooth boundary, and let
\begin{equation*}
0< \mu^2_1 < \mu^2_2 \le \mu_3^2 \le \cdots \nearrow \infty
\end{equation*}
be the eigenvalues (counted with multiplicity) of the Dirichlet Laplacian associated with $D$. Weyl~\cite{weyl11:1912} initiated the study of the asymptotic behavior of the eigenvalue counting function
\begin{equation*}
\mathscr{N}_D(\mu)=\#\left\{k\in\mathbb{N} : \mu_k^2\leq \mu^2 \right\}. 
\end{equation*}

In this paper we study the Euclidean ball $\mathscr{B}=\{x\in\mathbb{R}^d : |x|<R \}$ with $R<\infty$. It is well-known that
\begin{equation*}
\mathscr{N}_\mathscr{B}(\mu)=(2\pi)^{-d}\omega_d \vol(\mathscr{B})\mu^d-4^{-1}(2\pi)^{-d+1}\omega_{d-1}\vol(\partial \mathscr{B})\mu^{d-1}+\mathscr{R}_\mathscr{B}(\mu)
\end{equation*}
with the remainder satisfying
\begin{equation}
\mathscr{R}_\mathscr{B}(\mu)=o(\mu^{d-1}), \label{known-bound}
\end{equation}
where $\omega_k=$ the volume of the unit ball in $\mathbb{R}^k$. See Ivrii~\cite{Ivrii1980} and Melrose~\cite{Mel:1980} for results on more general manifolds with boundary.

In dimension two, there are better bounds of $\mathscr{R}_\mathscr{B}(\mu)$.  The bound $O(\mu^{2/3})$ was given in Kuznecov and Fedosov~\cite{kuz:1965} and Colin de Verdi\`ere~\cite{colin:2011}, which was improved a little bit in \cite{GWW2018} and recently to $O(\mu^{131/208}(\log \mu)^{18627/8320})$ in \cite{GMWW:2019}. We guess that $O_{\epsilon}(\mu^{1/2+\epsilon})$ may be its true size.

The aim of this paper is to improve the bound \eqref{known-bound} in all dimensions greater than two.
\begin{theorem}\label{ball-thm}
\begin{equation*}
\mathscr{R}_\mathscr{B}(\mu)=O\left(\mu^{d-2+\frac{131}{208}}(\log \mu)^{\frac{18627}{8320}}\right).
\end{equation*}
\end{theorem}

For results on other interesting (planar) domains, see \cite{Kuznecov:1965} for ellipses, \cite{GMWW:2019} for annuli, \cite{Kuznecov:1966} for domains of separable variables type, etc.

We follow the strategy used in Colin de Verdi\`ere~\cite{colin:2011} to prove Theorem \ref{ball-thm}--first establish the correspondence between eigenvalues and certain lattice points and then count the number of lattice points instead. We notice that previously such a strategy was mainly used to study two dimensional domains like planar disks, annuli, etc. Our work below shows that it also works well for high dimensional domains. A new difficulty is that we need to count lattice points with different weights originated from the multiplicities of eigenvalues. We overcome it by decomposing this new lattice counting problem into finitely many standard lattice counting problems associated with domains of decreasing sizes.

An interesting weighted lattice point problem was recently studied in Iosevich and Wyman \cite{IW19} for their study of the Weyl law for products of manifolds without boundary (e.g. spheres). The authors adjusted the standard method (of using oscillatory integrals and the Poisson summation formula) according to the weights. In particular, they mollify the product of the characteristic function of a ball and the homogeneous weight function and estimate the size of the Fourier transform of such a mollified product.

We remark that it is possible to improve the bound in Theorem \ref{ball-thm} by applying the decoupling techniques in harmonic analysis. They are emerging powerful tools which have many applications especially in number theoretical problems. See for example the preprint \cite{BWpreprint} by Bourgain and Watt in which they proved (among others) improved estimates for the Dirichlet divisor and Gauss circle problems.

For functions $f$ and $g$ with $g$ taking nonnegative real values,
$f\lesssim g$ means $|f|\leqslant Cg$ for some constant $C$. If $f$
is nonnegative, $f\gtrsim g$ means $g\lesssim f$. The Landau
notation $f=O(g)$ is equivalent to $f\lesssim g$. The notation
$f\asymp g$ means that $f\lesssim g$ and $g\lesssim f$.

The rest of this paper is organized as follows. In Section \ref{sec2} we provide approximations of zeros of the Bessel function of the first kind and nonnegative order. It generalizes the results in Colin de Verdi\`ere~\cite[Section 3]{colin:2011} where the order of the Bessel function is assumed to be nonnegative integers. In Section \ref{sec3} we deal with the aforementioned difficulty and prove that the eigenvalue counting problem can be indeed reduced to finitely many similar lattice counting problems. In Section \ref{sec4} we resolve the latter problems to obtain the main theorem.


\section{Zeros of the Bessel function $J_{\nu}(x)$} \label{sec2}

Throughout this paper we use
\begin{equation}
g(t)=\left(\sqrt{1-t^2}-t\arccos t\right)/\pi,\quad t\in [0,1].  \label{def-g}
\end{equation}

There are enormous literature on the theory of Bessel functions (including their asymptotics and zeros). See for example \cite{watson:1966} and \cite[Chapter 9 and 10]{abram:1972}. For our study of Laplacian eigenvalues of balls we need the following two lemmas on asymptotics of the Bessel function of the first kind $J_{\nu}$. They can be easily proved by using the method of stationary phase or Olver's asymptotics of Bessel functions.

\begin{lemma} \label{lemma1}
For any $c>0$ and $\nu\geq 0$, if $x\geq \max\{(1+c)\nu, 2\}$ then
\begin{equation}
J_{\nu}(x)=\frac{(2/\pi)^{1/2}}{(x^2-\nu^2)^{1/4}} \left(\sin\left( \pi
x g\left(\frac{\nu}{x}\right)+\frac{\pi}{4}\right)+O_c\left(x^{-1}\right)\right).\label{stat-pha-0}
\end{equation}
\end{lemma}

\begin{proof}
When $\nu$ is a nonnegative integer, a proof is provided in the appendix of \cite{GMWW:2019}. For real positive $\nu$ we recall the integral representation of $J_{\nu}(x)$ for $x>0$ (\cite[p.360]{abram:1972})
\begin{equation}
J_{\nu}(x)=\frac{1}{\pi}\int_0^\pi \cos (x\sin\theta-\nu \theta) \,\textrm{d}\theta-\frac{\sin(\nu \pi)}{\pi}\int_0^\infty e^{-x\sinh t-\nu t} \,\textrm{d}t. \label{bessel-fcn}
\end{equation}

We first study the integral
\begin{equation*}
\int_0^\pi e^{i x \phi(\theta)} \,\textrm{d}\theta \quad \textrm{with} \quad \phi(\theta)=\sin\theta-\frac{\nu}{x}\theta
\end{equation*}
and then take the real part to obtain the first integral in \eqref{bessel-fcn}. The phase function $\phi$ has only one critical point  $\beta:=\arccos(\nu/x)$ in $[0, \pi]$. Applying the method of stationary phase in a sufficiently small neighborhood of $\beta$ yields the contribution
\begin{equation}
\left(2\pi\right)^{1/2}|\phi''(\beta)|^{-1/2} e^{i(x\phi(\beta)-\pi/4)}x^{-1/2}+O_c(x^{-3/2}).\label{stat-pha-2}
\end{equation}
Applying integration by parts twice over the domain away from $\beta$ yields the contribution of the real part   \begin{equation}
\frac{\sin(\nu\pi)}{x+\nu}+O_c(x^{-2}).\label{stat-pha-3}
\end{equation}

We next show that
\begin{equation}
\int_0^\infty e^{-x\sinh t-\nu t}  \,\textrm{d}t=\frac{1}{x+\nu}+O\left(\frac{1}{(x+\nu)^{3}}\right).\label{stat-pha-1}
\end{equation}
Indeed, by using the Maclaurin series of $\sinh t$ and changing variables $s=(x+\nu)t$ we have
\begin{equation*}
\int_0^\infty e^{-x\sinh t-\nu t}  \,\textrm{d}t=\frac{1}{x+\nu}\int_0^\infty e^{-s}e^{-x\sigma\left(\frac{s}{x+\nu} \right)}  \,\textrm{d}s,
\end{equation*}
where
\begin{equation*}
\sigma(t)=\sum_{k=1}^{\infty} \frac{t^{2k+1}}{(2k+1)!}.
\end{equation*}
Therefore, by the mean value theorem, we get
\begin{equation*}
\left|\int_0^\infty e^{-x\sinh t-\nu t}  \,\textrm{d}t-\frac{1}{x+\nu}\right|\leq \frac{x}{x+\nu} \int_0^\infty e^{-s}\sigma\left(\frac{s}{x+\nu} \right)  \,\textrm{d}s.
\end{equation*}
Simplifying the right hand side with the gamma function yields a bound $O((x+\nu)^{-3})$. This proves \eqref{stat-pha-1}.

Finally, using \eqref{stat-pha-2}, \eqref{stat-pha-3} and \eqref{stat-pha-1} we get the desired asymptotics.
\end{proof}

\begin{lemma}\label{lemma2}
For any $c>0$ and all sufficiently large $\nu$, if $\nu<x\leq (1+c)\nu$  then
\begin{equation}
J_{\nu}(x)=\frac{\left(12\pi x g\left(\nu/x \right)\right)^{1/6}}{\left(x^2-\nu^2\right)^{1/4}}
\left(\mathrm{Ai}\left(-\left(\frac{3\pi}{2} x g\left(\frac{\nu}{x} \right)\right)^{2/3}\right)+E_{\nu}(x)\right),\label{olver3}
\end{equation}
where Ai denotes the Airy function of the first kind and
\begin{equation}
E_{\nu}(x)=O\left(\nu^{-4/3}\left(1+\left(x g\left(\frac{\nu}{x} \right)\right)^{1/6}\right)\right).\label{olver2}
\end{equation}
If we further assume that $x g(\nu/x)\geq 1$ then
\begin{equation}
J_{\nu}(x)=\frac{(2/\pi)^{1/2}}{\left(x^2-\nu^2\right)^{1/4}}
\left(\sin\left( \pi x g\left(\frac{\nu}{x}\right)+\frac{\pi}{4}\right)+O\left(\left(x g\left(\frac{\nu}{x}\right)\right)^{-1} \right)\right).\label{olver4}
\end{equation}
\end{lemma}

\begin{proof}
Let $x=\nu z$. For sufficiently large $\nu$ we will apply Olver's asymptotics of $J_{\nu}(x)$, \eqref{olver1}, with $\zeta=\zeta(z)$ determined by \eqref{def-zeta1}. It is easy to check that $0<-\zeta\lesssim_c 1$ and
\begin{equation*}
\nu^{2/3}\zeta=-\left(\frac{3\pi}{2} x g\left(\frac{\nu}{x} \right)\right)^{2/3}.
\end{equation*}
Hence
\begin{equation*}
J_{\nu}(x)=\frac{\left(12\pi x g\left(\nu/x \right)\right)^{1/6}}{\left(x^2-\nu^2\right)^{1/4}}
\left(\mathrm{Ai}\left(\nu^{2/3}\zeta\right)+E_{\nu}(x)\right),
\end{equation*}
where
\begin{equation*}
E_{\nu}(x)=\mathrm{Ai}\left(\nu^{2/3}\zeta\right) O\left(\nu^{-2} \right)+\mathrm{Ai}'\left(\nu^{2/3}\zeta \right)O\left(\nu^{-4/3}\right).
\end{equation*}
Using $\mathrm{Ai}(-r)=O(1)$ and $\mathrm{Ai}'(-r)=O(r^{1/4})$ (see \cite[p.448--449]{abram:1972}), we readily get \eqref{olver2}. This proves the first part of the lemma.

The asymptotics \eqref{olver4} follows directly from \eqref{olver3} and the well-known formula
\begin{equation*}
\mathrm{Ai}(-r)=\pi^{-1/2}r^{-1/4}\left(\sin\left(\frac{2}{3}r^{3/2}+\frac{\pi}{4}\right)+O\left(r^{-3/2}\right)\right), \quad r>0
\end{equation*}
(see for example \cite[p.448]{abram:1972}).
\end{proof}

Since we want to study the spectrum of the Dirichlet Laplacian, we only need to consider positive zeros of the Bessel function $J_{\nu}$. For nonnegative $\nu$ let us denote the $k$th positive (simple) zero of $J_{\nu}$ by $j_{\nu, k}$. It is known that $\nu<j_{\nu,1}<j_{\nu,2}<j_{\nu,3}<\cdots$. See \cite[p.370]{abram:1972}.

\begin{proposition}\label{large-nu}
For any $c>0$ and all sufficiently large $\nu$ the zeros, $\{j_{\nu, k}\}_{k=1}^{\infty}$, satisfy the following:
\begin{enumerate}
\item if $j_{\nu,k}\geq (1+c)\nu$ then
\begin{equation}
j_{\nu,k} g\left(\frac{\nu}{j_{\nu,k}}\right)=k-\frac{1}{4}+O\left(\frac{1}{\nu+k}\right);\label{appro1}
\end{equation}

\item if $\nu<j_{\nu,k}<(1+c)\nu$ then
\begin{equation}
j_{\nu,k} g\left(\frac{\nu}{j_{\nu,k}}\right)=k-\frac{1}{4}+O\left(\frac{1}{k}\right).\label{appro2}
\end{equation}
\end{enumerate}
\end{proposition}

\begin{proof}
For any fixed $c>0$ and any fixed sufficiently large $\nu$ we will use Lemma \ref{lemma1} and \ref{lemma2} to study zeros of $J_{\nu}(x)$ in the interval $(\nu, \pi(K+\frac{1}{2}\nu+\frac{1}{2}))$ for integer $K\rightarrow \infty$. In view of those asymptotics, we need to study the function $h_{\nu}(x):=xg(\nu/x)$ which is mapping from $(\nu, \pi(K+\frac{1}{2}\nu+\frac{1}{2}))$ onto $(0, K+\frac{1}{2}+O_{\nu}((K+\nu)^{-1}))$, continuous and strictly increasing. Hence for each integer $1\leq k\leq K$ there exists an interval $(a_{\nu, k}, b_{\nu, k})\subset (\nu, \pi(K+\frac{1}{2}\nu+\frac{1}{2}))$ such that $h_{\nu}$ maps $(a_{\nu, k}, b_{\nu, k})$ to $(k-3/8, k)$ bijectively.

If $\nu$ is sufficiently large then for each $1\leq k\leq K$
\begin{equation}
J_{\nu}(a_{\nu, k})J_{\nu}(b_{\nu, k})<0. \label{IVT-cond}
\end{equation}
This is a consequence of the asymptotics \eqref{stat-pha-0}, \eqref{olver3} and \eqref{olver4}. Heuristically \eqref{IVT-cond} follows from the facts that the sine and Airy functions oscillate around zero, $h_{\nu}$ is monotone, the error terms in the aforementioned asymptotics are small and the intervals $(a_{\nu, k}, b_{\nu, k})$ are properly chosen (to contain zeros).

To prove \eqref{IVT-cond} rigorously, we fix a sufficiently large constant $C>10$ such that if $h_{\nu}(x)>C$ then the error term $O(h_{\nu}(x)^{-1})$ in \eqref{olver4} is less than $10^{-10}$. We only consider sufficiently large $\nu$ such that the error term $O(x^{-1})$ in \eqref{stat-pha-0} is less than $10^{-10}$.
If
\begin{equation}
h_{\nu}\left(\left(a_{\nu, k}, b_{\nu, k}\right)\right)=(k-3/8,k)\subset (0, 2C] \label{case1}
\end{equation}
we use \eqref{olver3} to prove \eqref{IVT-cond}, otherwise we use \eqref{stat-pha-0} and \eqref{olver4}. In the former case when \eqref{case1} holds, there are at most $\lfloor 2C\rfloor$ choices of $k$. We denote by $t_k$ ($k\in\mathbb{N}$) the $k$th zero of the equation $\textrm{Ai}(-x)=0$. The results in \cite[p.405]{olver:1997} implies that $t_k$ is contained in the interval
\begin{equation*}
\left(\left(\frac{3\pi}{2}\left(k-0.36\right)\right)^{2/3}, \left(\frac{3\pi}{2}\left(k-0.14\right)\right)^{2/3} \right)
\end{equation*}
which is a proper subset of
\begin{equation*}
\left(\left(\frac{3\pi}{2}h_{\nu}(a_{\nu, k})\right)^{2/3}, \left(\frac{3\pi}{2}h_{\nu}(b_{\nu, k})\right)^{2/3} \right).
\end{equation*}
Hence
\begin{equation*}
\mathrm{Ai}\left(-\left(\frac{3\pi}{2} h_{\nu}(a_{\nu, k}) \right)^{2/3}\right)\mathrm{Ai}\left(-\left(\frac{3\pi}{2} h_{\nu}(b_{\nu, k}) \right)^{2/3}\right)<0.
\end{equation*}
Since the errors $E_{\nu}(a_{\nu, k})$ and $E_{\nu}(b_{\nu, k})$ in \eqref{olver3} are both of size $O_C(\nu^{-3/4})$, if $\nu$ is sufficiently large then \eqref{IVT-cond} follows easily. In the latter case when \eqref{case1} fails, we must have $h_{\nu}((a_{\nu, k}, b_{\nu, k}))\subset (C, \infty)$. Thus \eqref{IVT-cond} follows immediately from \eqref{stat-pha-0} and \eqref{olver4} if we notice that
\begin{equation*}
\sin\left( \pi h_{\nu}(a_{\nu, k})+\frac{\pi}{4}\right)\sin\left( \pi h_{\nu}(b_{\nu, k})+\frac{\pi}{4}\right)=-2^{-1/2}\sin\left(\frac{\pi}{8}\right)<0.
\end{equation*}

Let us now continue to prove the proposition. By the intermediate value theorem, \eqref{IVT-cond} implies that there exists at least one zero of $J_{\nu}$ in $(a_{\nu, k}, b_{\nu, k})$  for each $1\leq k\leq K$. On the other hand as a consequence of the McMahon's expansion of large zeros of $J_{\nu}$
\begin{equation*}
j_{\nu,K}=\pi\left(K+\frac{1}{2}\nu-\frac{1}{4} \right)+O_{\nu}\left(\frac{1}{K}\right)
\end{equation*}
(see \cite[p.371]{abram:1972}), there are exactly $K$ zeros in the interval $(\nu, \pi(K+\frac{1}{2}\nu+\frac{1}{2}))$. Hence there exists one and only one zero in each $(a_{\nu, k}, b_{\nu, k})$, i.e. $j_{\nu,k}$, satisfying
\begin{equation}
h_{\nu}(j_{\nu,k})-k+\frac{1}{4}\in \left(-\frac{1}{8}, \frac{1}{4}\right).\label{appro3}
\end{equation}

If $j_{\nu,k}\geq (1+c)\nu$ then \eqref{appro1} follows from \eqref{stat-pha-0}, \eqref{appro3} and the bound $j_{\nu,k}>\sqrt{\nu^2+\pi^2(k-1/4)^2}$ (see McCann~\cite[p.102]{McCann:1977}). If $\nu<j_{\nu,k}<(1+c)\nu$ and $h_{\nu}(j_{\nu,k})>C$ then \eqref{olver4} together with \eqref{appro3} gives
\begin{equation*}
h_{\nu}(j_{\nu,k})=k-\frac{1}{4}+O\left(h_{\nu}(j_{\nu,k})^{-1}\right).
\end{equation*}
This formula itself implies $h_{\nu}(j_{\nu,k})\asymp k$ hence \eqref{appro2}. If $\nu<j_{\nu,k}<(1+c)\nu$ and $h_{\nu}(j_{\nu,k})\leq C$ then $1\leq k<C+1$. Hence \eqref{appro2} follows trivially from \eqref{appro3}.
\end{proof}

\begin{remark}\label{rm1}
The error terms in \eqref{appro1} and \eqref{appro2} are of course small when $k$ is large. However they are quite small even when $k$ is small. In fact the proof above implies that the error in \eqref{appro1} is less than $10^{-10}$ while the one in \eqref{appro2} is less than $10^{-10}$ if $k\geq C+1$ and $1/4$ if $1\leq k<C+1$.
\end{remark}

\begin{proposition}\label{small-nu}
For any $V>0$ there exists a constant $K>0$ such that if $0\leq \nu\leq V$ and $k\geq K$ then the zeros $j_{\nu,k}$ satisfy
\begin{equation}
j_{\nu,k} g\left(\frac{\nu}{j_{\nu,k}}\right)=k-\frac{1}{4}+O_V\left(\frac{1}{\nu+k}\right).\label{appro4}
\end{equation}
\end{proposition}

\begin{proof}
For any $V>0$, if $0\leq \nu\leq V$ and $k$ is sufficiently large then the McMahon's expansion (\cite[p.371]{abram:1972}) gives
\begin{equation*}
j_{\nu,k}=\pi\left(k+\frac{1}{2}\nu-\frac{1}{4} \right)+O_{V}\left(\frac{1}{k+\nu}\right).
\end{equation*}
Taylor's formula of $g$ at $0$ also gives
\begin{equation*}
j_{\nu,k} g\left(\frac{\nu}{j_{\nu,k}}\right)=\frac{1}{\pi}j_{\nu,k}-\frac{1}{2}\nu+O_V\left(\frac{1}{j_{\nu,k}}\right).
\end{equation*}
Combining the above two formulas gives \eqref{appro4}.
\end{proof}

Finally we are ready to prove the following approximations of zeros $j_{\nu, k}$. The key point is that these approximations are good for essentially all (not just large) $\nu\geq 0$ and $k\in \mathbb{N}$ with relatively small errors!

Let $F: [0, \infty)\times [0, \infty)\setminus \{O\}\rightarrow \mathbb{R}$ be the function homogeneous of degree $1$ such that $F\equiv1$ on the graph of $g$. In fact, $F$ is the Minkowski functional of the domain $\mathcal{D}$ (the shaded area in Figure \ref{domainD} bounded by axes and the graph of $g$).

\begin{figure}[ht]
\centering
\includegraphics[width=0.5\textwidth]{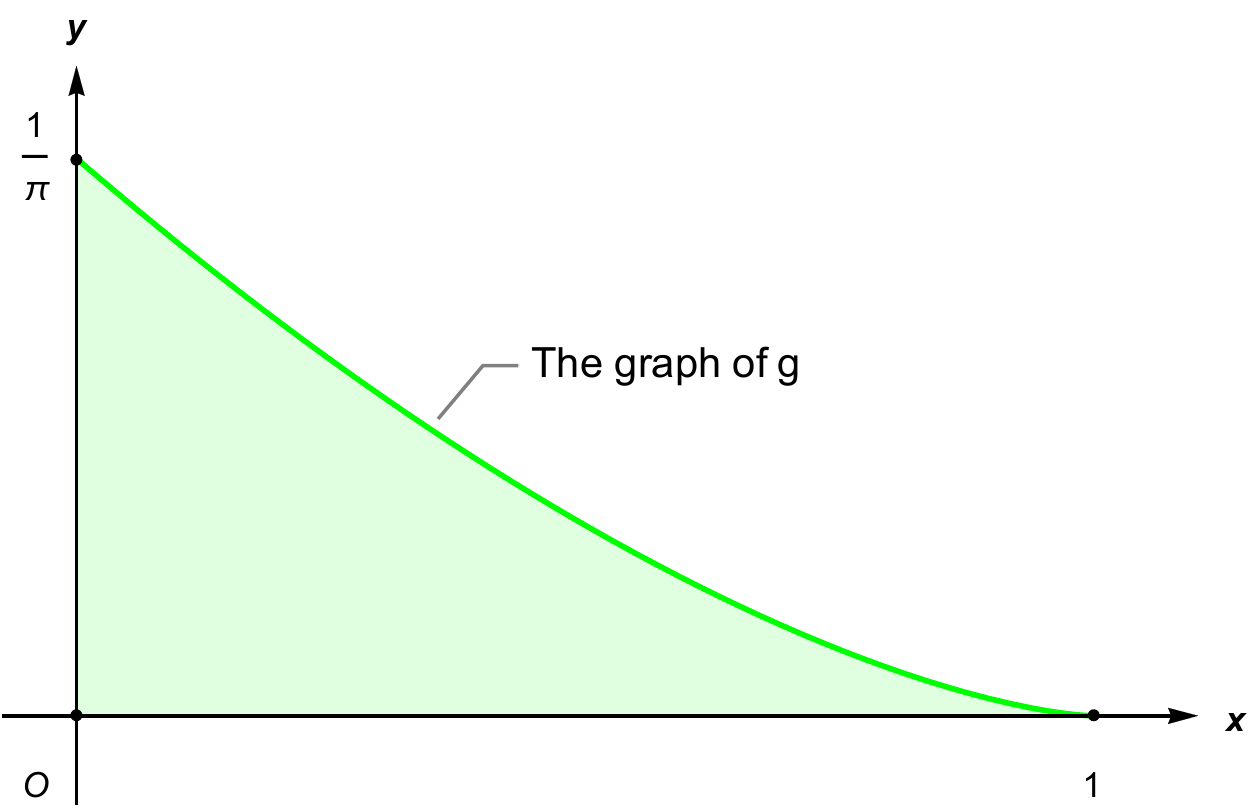}  
\caption{The domain $\mathcal{D}$.}
\label{domainD}
\end{figure}

\begin{theorem}\label{approximation}
There exists a small constant $c>0$ and a $N\in\mathbb{N}$ such that if $\nu>N$ then
\begin{equation}
j_{\nu, k}=F(\nu,k-1/4)+R_{\nu,k},\label{approximation0}
\end{equation}
where
\begin{equation}
R_{\nu,k}=\left\{
            \begin{array}{ll}
            O\left((\nu+k)^{-1}\right),                                         & \textrm{if $j_{\nu,k}\geq (1+c)\nu$,}\\
            O\left(\nu^{1/3}k^{-4/3}\right),                                   &\textrm{if $j_{\nu,k}<(1+c)\nu$.}
            \end{array}
        \right. \label{approximation1}
\end{equation}

If $0\leq \nu\leq N$ and $k$ is sufficiently large then \eqref{approximation0} holds with
\begin{equation}
R_{\nu,k}=O\left((\nu+k)^{-1}\right). \label{approximation2}
\end{equation}
\end{theorem}

\begin{proof}
Applying Proposition  \ref{large-nu} with a small constant $c$ (to be determined below) and all sufficiently large $\nu$ yields that
\begin{equation*}
\frac{k-1/4+O(1)}{\nu}=\frac{g\left(\nu/j_{\nu,k}\right)}{\nu/j_{\nu,k}}+O\left(\nu^{-1}\right).
\end{equation*}
Since $g(t)/t$ is strictly decreasing, if $j_{\nu,k}<(1+c)\nu$ and $\nu$ is sufficiently large then
\begin{equation}
\frac{k-1/4+O(1)}{\nu}\leq \frac{2g\left((1+c)^{-1}\right)}{(1+c)^{-1}} \label{slope}
\end{equation}
which goes to $0$ as $c\rightarrow 0$. Observe that if $0<y\leq c'x$ with $c'$ small enough then $\partial_y F(x,y)=O(x^{1/3}y^{-1/3})$, otherwise $\partial_y F(x,y)=O(1)$. We choose $c$ so small that the right hand side of \eqref{slope} is less than $c'$. Then combining Proposition \ref{large-nu} (with Remark \ref{rm1}), properties of $F$ and the mean value theorem yields the bound in \eqref{approximation1} when $j_{\nu,k}<(1+c)\nu$. The same argument proves \eqref{approximation1} when $j_{\nu,k}\geq (1+c)\nu$ if we note that
\begin{equation*}
\frac{k-1/4+O(1)}{\nu}\geq \frac{g\left((1+c)^{-1}\right)}{2(1+c)^{-1}}
\end{equation*}
whenever $\nu$ is sufficiently large.

The desired result for ``small'' $\nu$'s follows similarly from Proposition \ref{small-nu} and the mean value theorem.
\end{proof}


\section{Spectrum counting to lattice counting} \label{sec3}

To study the eigenvalue counting function $\mathscr{N}_\mathscr{B}(\mu)$, by a simple renormalization, we may assume $\mathscr{B}=B(0,1)$ from now on. We search for bounds of
\begin{equation*}
\mathscr{R}_\mathscr{B}(\mu)=\mathscr{N}_\mathscr{B}(\mu)-2^{-d}\Gamma\left(\frac{d}{2}+1\right)^{-2} \mu^d + \frac{1}{2\cdot(d-1)!}\mu^{d-1}, \quad d\geq 3.
\end{equation*}

It can be checked by standard separation of variables that for $d\geq 3$ the spectrum of the Dirichlet Laplacian associated with $\mathscr{B}$ consists of the numbers $x_{n,k}^2$, $n\in\mathbb{Z}_{+}:=\mathbb{N}\cup\{0\}$, $k\in\mathbb{N}$, where $x_{n,k}$'s are the positive zeros of the (ultra)spherical Bessel function $x^{1-d/2}J_{n+\frac{d}{2}-1}(x)$. In another word
\begin{equation*}
x_{n,k}=j_{n+\frac{d}{2}-1,k}
\end{equation*}
to which we can apply Theorem \ref{approximation} with $\nu=n+d/2-1$. For each pair $(n,k)$ the $x_{n,k}^2$ appears $m_n^d:=\binom{n+d-1}{d-1}-\binom{n+d-3}{d-1}$ \footnote{We follow the convention that $\binom{l}{m}=0$ if $m>l$.}times in the spectrum. In particular $m_0^d=1$ and $m_1^d=d$. See Gurarie~\cite[\S4.5]{Gur:1992}.

We observe that Theorem \ref{approximation} roughly tells us that each $x_{n,k}\leq \mu$ corresponds to a point $(n+d/2-1, k-1/4)\in \mu\mathcal{D}$. We also observe that the multiplicity $m_n^d$ is strictly increasing in $n$. These motivate the following definitions. For $l\in \mathbb{Z}_+$,
\begin{equation*}
\mathscr{N}_l(\mu):=\#\{(n,k)\in \mathbb{Z}_+\times\mathbb{N} : x_{n,k}\leq \mu, n\geq l\}\footnote{The multiplicity of $x_{n,k}$ is not counted in the definition of this set.}
\end{equation*}
and
\begin{equation*}
\mathcal{N}_l(\mu):=\#\left\{ \left(n+\frac{d}{2}-1, k-\frac{1}{4}\right)\in \mu\mathcal{D} : n\in \mathbb{Z}_+, n\geq l, k\in\mathbb{N}\right\}.
\end{equation*}
Then
\begin{equation*}
\mathscr{N}_\mathscr{B}(\mu)=\sum_{l=0}^{\infty} \left(m_{l}^d-m_{l-1}^d\right)\mathscr{N}_l(\mu)
\end{equation*}
with $m_{-1}^d:=0$. Correspondingly we define
\begin{equation}
\mathcal{N}(\mu):=\sum_{l=0}^{\infty}\left(m_{l}^d-m_{l-1}^d\right)\mathcal{N}_l(\mu).\label{lattice}
\end{equation}
It is easy to observe that the summands in the above two sums are equal to zero if $l\geq \mu-\frac{d-2}{2}$, and that
\begin{align}
m_{l}^d-m_{l-1}^d&=\binom{l+d-2}{d-2}-\binom{l+d-4}{d-2},\label{bino}\\
                 &=O_d\left(l^{d-3}+1\right),\quad \textrm{if $d\geq 3$, $l\in\mathbb{Z}_+$}.\label{bino-bound}
\end{align}

\begin{lemma}
There exists a constant $C>0$ such that for $l\in \mathbb{Z}_+$
\begin{equation*}
\left|\mathscr{N}_l(\mu)-\mathcal{N}_l(\mu)\right|\leq \mathcal{N}_l(\mu+C\mu^{-3/7})-\mathcal{N}_l(\mu-C\mu^{-3/7})+O\left(\mu^{4/7}\right).
\end{equation*}
\end{lemma}

\begin{proof}
In the following argument $\max\{n,k\}$ is implicitly assumed to be sufficiently large. This treatment will add an $O(1)$ error to the desired inequality. Define $\nu=n+d/2-1$ for short and for $k\in\mathbb{N}$
\begin{equation}
\mathscr{N}_{l,k}(\mu):=\#\{n\in \mathbb{Z}_+ : j_{\nu,k}\leq \mu, n\geq l\}\label{countfcn1}
\end{equation}
and
\begin{equation}
\mathcal{N}_{l,k}(\mu):=\#\left\{ n\in \mathbb{Z}_+ : \left(\nu, k-1/4\right)\in \mu\mathcal{D}, n\geq l\right\}.\label{countfcn2}
\end{equation}
Hence $\mathscr{N}_l(\mu)$ and $\mathcal{N}_l(\mu)$ are sums of \eqref{countfcn1} and \eqref{countfcn2} respectively over integer $1\leq k<\mu$. Using Theorem \ref{approximation} and properties of $F$ yields
\begin{align}
&\left|\mathscr{N}_{l,k}(\mu)-\mathcal{N}_{l,k}(\mu) \right|\leq \nonumber\\
&\quad \#\left\{n\in\mathbb{Z}_+ : \mu-|R_{\nu,k}|\leq F(\nu,k-1/4)\leq\mu+|R_{\nu,k}|, n\geq l \right\}.\label{set}
\end{align}

We have the following bounds, summing which over $k$ yields the desired inequality. If $1\leq k\leq \mu^{1/4}$ then
\begin{equation*}
\left|\mathscr{N}_{l,k}(\mu)-\mathcal{N}_{l,k}(\mu) \right|\lesssim \mu^{1/3}k^{-4/3};
\end{equation*}
if $\mu^{1/4}<k\leq \mu^{4/7}$ then
\begin{equation*}
\left|\mathscr{N}_{l,k}(\mu)-\mathcal{N}_{l,k}(\mu) \right|\lesssim 1;
\end{equation*}
if $\mu^{4/7}<k<\mu$ there exists a constant $C>0$ such that
\begin{equation*}
\left|\mathscr{N}_{l,k}(\mu)-\mathcal{N}_{l,k}(\mu) \right|\leq \mathcal{N}_{l,k}(\mu+C\mu^{-3/7})-\mathcal{N}_{l,k}(\mu-C\mu^{-3/7}).
\end{equation*}
We prove these bounds by estimating the size of \eqref{set}. Note that the definitions of \eqref{countfcn1} and \eqref{countfcn2} ensure that we only need to consider $\nu\leq \mu$. Hence $F(\nu,k-1/4)=\mu+O(\mu^{1/3})$, which implies $|(\nu,k-1/4)|\asymp \mu$. Therefore $(\nu+k)^{-1}\asymp \mu^{-1}$ and we always have $R_{\nu,k}=O(\mu^{1/3}k^{-4/3})$. If $1\leq k\leq \mu^{4/7}$ the bounds follow from the mean value theorem and the fact $\partial_x F(x,y)\asymp_c 1$ if $0<y\leq cx$ for any $c>0$. If $\mu^{4/7}<k<\mu$ then $R_{\nu,k}=O(\mu^{-3/7})$ and the last bound follows readily.
\end{proof}

Collecting the above definitions, observations and estimates easily yields
\begin{theorem}\label{comparison}
There exists a constant $C>0$ such that
\begin{equation*}
\left|\mathscr{N}_\mathscr{B}(\mu)-\mathcal{N}(\mu)\right|\leq  \mathcal{N}(\mu+C\mu^{-3/7})-\mathcal{N}(\mu-C\mu^{-3/7})+O\left(\mu^{d-2+4/7}\right).
\end{equation*}
\end{theorem}


\section{Proof of The Main Theorem} \label{sec4}

Theorem \ref{ball-thm} follows immediately from Theorem \ref{comparison} and \ref{thm4}. To prove Theorem \ref{thm4} we need the following result on counting lattice points in subsets of the enlarged domain $\mu\mathcal{D}$.

\begin{lemma}\label{lattice-number}
For integer $0\leq l<\mu-d/2+1$, let $(\mu\mathcal{D})_l=(\mu\mathcal{D})\cap \{(x,y)\in\mathbb{R}^2 : x\geq l+\frac{d-3}{2} \}$ be a subset of $\mu\mathcal{D}$. Then
\begin{equation}
\mathcal{N}_l(\mu)=\vol\left((\mu\mathcal{D})_l\right)-\frac{1}{4}\left(\mu-l-\frac{d-3}{2}\right)+
O\left(\mu^{\frac{131}{208}}(\log \mu)^{\frac{18627}{8320}}\right). \label{111}
\end{equation}
\end{lemma}

\begin{remark}
This result is already (essentially) obtained in \cite[Section 4]{GMWW:2019} by using M.N. Huxley's deep results in \cite{Huxley:2003} on bounds for rounding error sums. In fact \cite[Theorem 4.1]{GMWW:2019} deals with the lattice counting in a more complicated domain. Our domain $\mathcal{D}$ here is simply a special case of it (with $r=0$ and $R=1$ there). To prove this lemma, we just need to repeat the computation of $\mathcal{N}_{\mathcal{D}_1}^{u}(\mu)$ in the proof of \cite[Theorem 4.1]{GMWW:2019}. We give a sketch in the appendix.

We observe that the bound $O(\mu^{131/208}(\log \mu)^{18627/8320})$ may be improved for large $l$'s. This is particularly clear as $l\rightarrow \mu-d/2+1$ when the two main terms in \eqref{111} are much smaller than the error. However, since we are not able to essentially improve the bound for (say) $l<\mu/2$, improvements for large $l$'s do not lead to a better remainder estimate in the following theorem.  Hence we are satisfied with the current form of \eqref{111}.
\end{remark}

\begin{theorem}\label{thm4}
\begin{equation*}
\mathcal{N}(\mu)=2^{-d}\Gamma\left(\frac{d}{2}+1\right)^{-2} \mu^d-\frac{1}{2\cdot(d-1)!}\mu^{d-1}+O\left( \mu^{d-2+\frac{131}{208}}(\log \mu)^{\frac{18627}{8320}}\right).
\end{equation*}
\end{theorem}

\begin{proof}
By \eqref{lattice} and Lemma \ref{lattice-number}, we split $\mathcal{N}(\mu)$ into three parts and evaluate them one by one.

First,
\begin{equation}
\sum_{0\leq l\leq \mu-\frac{d-2}{2}} \left(m_{l}^d-m_{l-1}^d\right)\vol\left((\mu\mathcal{D})_l\right)=2^{-d}\Gamma\left(\frac{d}{2}+1\right)^{-2} \mu^d+O\left(\mu^{d-2} \right). \label{mainterm}
\end{equation}
To verify this asymptotics we apply to the left side that for $l\geq 1$
\begin{equation*}
m_{l}^d-m_{l-1}^d=\frac{2}{(d-3)!}\left(l+\frac{d-3}{2}\right)^{d-3}+R_l,
\end{equation*}
where $R_l=0$ if $d=3, 4, 5$ and $R_l=O(l^{d-5})$ if $d\geq 6$, and that $m_{1}^d-m_{-1}^d=1$. Hence the left side
of \eqref{mainterm} is equal to
\begin{equation}
\vol\left((\mu\mathcal{D})_0\right)+\sum_{0<l\leq \lfloor\mu-\frac{d-2}{2}\rfloor}f(l)+\sum_{0< l\leq \lfloor\mu-\frac{d-2}{2}\rfloor} R_l\vol\left((\mu\mathcal{D})_l\right), \label{mainterm1}
\end{equation}
where
\begin{equation*}
f(x)=\frac{2}{(d-3)!}\left(x+\frac{d-3}{2}\right)^{d-3}\int_{x+\frac{d-3}{2}}^{\mu} \mu g\left(\frac{t}{\mu} \right)  \, \textrm{d}t,\quad \textrm{if $d\geq 4$},
\end{equation*}
and
\begin{equation*}
f(x)=2\int_{x}^{\mu} \mu g\left(\frac{t}{\mu} \right)  \, \textrm{d}t,\quad\textrm{if $d=3$}.
\end{equation*}
It is easy to get that $\vol((\mu\mathcal{D})_0)=|\mathcal{D}|\mu^2+O(\mu)$ and the third term in \eqref{mainterm1} is at most $O(\mu^{d-2})$. An application of the Euler-Maclaurin summation formula yields that the second term in \eqref{mainterm1} is equal to
\begin{equation}
\int_{0}^{\lfloor\mu-\frac{d-2}{2}\rfloor} f(t) \, \textrm{d}t+\int_{0}^{\lfloor\mu-\frac{d-2}{2}\rfloor} \psi(t)f'(t) \, \textrm{d}t+\frac{1}{2}f\left(\lfloor\mu-\frac{d-2}{2}\rfloor\right)-\frac{1}{2}f(0),\label{mainterm2}
\end{equation}
where $\psi(t)=t-\lfloor t\rfloor-1/2$. By changing variables it is not hard to get
\begin{align*}
\int_{0}^{\lfloor\mu-\frac{d-2}{2}\rfloor} f(t) \, \textrm{d}t&=\frac{2}{(d-2)!}\int_{0}^{1} t^{d-2}g(t)\, \textrm{d}t\mu^d+O\left(\mu^{d-2} \right) \\
     &=2^{-d}\Gamma\left(\frac{d}{2}+1\right)^{-2} \mu^d+O\left(\mu^{d-2} \right),
\end{align*}
where in the last equality we have used integration by parts twice and properties of the beta and the gamma functions. Since $f'$ is monotone (if $d=3$) or a linear combination of two (piecewise) monotone functions, by the second mean value theorem the second term in \eqref{mainterm2} is $O(\mu^{d-2})$. It is also easy to observe that the third term in \eqref{mainterm2} is $O(\mu^{d-2})$ and the last term in \eqref{mainterm2} is of size $O(\mu^2)$ if $d\geq 4$ and equal to $-|\mathcal{D}|\mu^2$ if $d=3$. Collecting all the above information yields
\eqref{mainterm}.

Second,
\begin{equation*}
-\frac{1}{4}\sum_{0\leq l\leq \mu-\frac{d-2}{2}}\left(m_{l}^d-m_{l-1}^d\right)\left(\mu-l-\frac{d-3}{2}\right)=-\frac{\mu^{d-1}}{2(d-1)!}+O\left( \mu^{d-2}\right).
\end{equation*}
Indeed, using \eqref{bino} and summation by parts yields that the left side is equal to
\begin{equation*}
-\frac{1}{2}\sum_{0\leq l\leq \mu-\frac{d+2}{2}}\binom{l+d-2}{d-2}+O\left( \mu^{d-2}\right)
\end{equation*}
which is simply
\begin{equation*}
-\frac{1}{2}\binom{\lfloor\mu-\frac{d+2}{2}\rfloor+d-1}{d-1}+O\left( \mu^{d-2}\right)=-\frac{\mu^{d-1}}{2(d-1)!}+O\left( \mu^{d-2}\right).
\end{equation*}

Third, by \eqref{bino-bound} the remainder estimate of $\mathcal{N}_l(\mu)$ leads to the desired one of $\mathcal{N}(\mu)$.
\end{proof}


\appendix

\section{Olver's asymptotic expansions of Bessel functions}

Here is one of Olver's uniform asymptotic expansions of Bessel functions of large order (see \cite[p.368]{abram:1972}):
\begin{equation}
J_{\nu}(\nu z)\sim \left(\frac{4\zeta}{1-z^2}\right)^{1/4}\left(\frac{\mathrm{Ai}(\nu^{2/3}\zeta)}{\nu^{1/3}}
\sum_{k=0}^{\infty}\frac{a_k(\zeta)}{\nu^{2k}}+\frac{\mathrm{Ai}'(\nu^{2/3}\zeta)}{\nu^{5/3}}
\sum_{k=0}^{\infty}\frac{b_k(\zeta)}{\nu^{2k}}\right)\label{olver1}
\end{equation}
as real $\nu\rightarrow \infty$, uniformly for $z\in (0, \infty)$, where $\zeta=\zeta(z)$ is determined by
\begin{equation}
\frac{2}{3}(-\zeta)^{3/2}=\int_{1}^z\frac{\sqrt{t^2-1}}{t}\,\mathrm{d}t=
\sqrt{z^2-1}-\arccos\left(\frac{1}{z}\right), \quad 1\leq z<\infty,\label{def-zeta1}
\end{equation}
and
\begin{equation*}
\frac{2}{3}\zeta^{3/2}=\int_{z}^1\frac{\sqrt{1-t^2}}{t}\,\mathrm{d}t=\ln \frac{1+\sqrt{1-z^2}}{z}-\sqrt{1-z^2}, \quad 0<z\leq 1.
\end{equation*}
Here the branches are chosen so that $\zeta$ is real when $z$ is positive. $\mathrm{Ai}$ denotes the Airy function of the first kind. For the definitions and sizes of the coefficients $a_k(\zeta)$'s and $b_k(\zeta)$'s see \cite[p.368--369]{abram:1972}. In particular $a_0(\zeta)=1$ and $b_0(\zeta)$ is bounded.

\section{Proof of Lemma \ref{lattice-number}}

Let $h=g^{-1}: [0, 1/\pi]\rightarrow [0,1]$ be the inverse function of $g$. It is easy to check that
\begin{equation*}
\left|h^{(j)}(y)\right|\asymp y^{2/3-j} \quad \textrm{for $0<y\leq 1/\pi$, $j=1,2,3$}.
\end{equation*}
For $0\leq l<\mu-d/2+1$ we have
\begin{equation*}
\mathcal{N}_l(\mu)=\sum_{0<k-\frac{1}{4}\leq \mu g\left(\frac{l+d/2-1}{\mu}\right)}\left( \left\lfloor \mu h\left(\frac{k-1/4}{\mu} \right)-\frac{d}{2}+1 \right\rfloor-l+1\right).
\end{equation*}
Using $\lfloor t\rfloor=(t-1/2)-\psi(t)$ with $\psi$ the row-of-teeth function, we split the above sum into two.

The first one is
\begin{equation}
\sum_{0<k-\frac{1}{4}\leq \mu g\left(\frac{l+d/2-1}{\mu}\right)}\left( \mu h\left(\frac{k-1/4}{\mu} \right)-l-\frac{d-3}{2}  \right). \label{sum1}
\end{equation}
Applying the Euler-Maclaurin summation formula  yields
\begin{equation*}
\eqref{sum1}=\vol\left((\mu\mathcal{D})_l\right)-\frac{1}{4}\left(\mu-l-\frac{d-3}{2}\right)+O\left(\mu^{1/3} \right).
\end{equation*}

The second one is
\begin{equation}
\sum_{0<k-\frac{1}{4}\leq \mu g\left(\frac{l+d/2-1}{\mu}\right)} \psi\left(\mu h\left(\frac{k-1/4}{\mu}\right)-\frac{d}{2}+1 \right).\label{sum2}
\end{equation}
For $k\leq V:=\mu^{131/208}$ a trivial estimate produces an $O(\mu^{131/208})$ contribution. The rest part of \eqref{sum2} is divided into sums of the form
\begin{equation}
\sum_{M\leq k\leq M'\leq 2M} \psi\left(\mu h\left(\frac{k-1/4}{\mu}\right)-\frac{d}{2}+1 \right),\label{sum3}
\end{equation}
where $M=2^jV\lesssim \mu$, $j=0,1,2,\ldots$. By using \cite[Proposition 3]{Huxley:2003} with
\begin{equation*}
F(x)=\left(\frac{\mu}{M}\right)^{2/3}h\left(\frac{M}{\mu} x-\frac{1/4}{\mu}\right)+\frac{1-d/2}{N}
\end{equation*}
and $N=M^{2/3}\mu^{1/3}$, we get
\begin{equation*}
\eqref{sum3}=\sum_{M\leq k\leq M'\leq 2M} \psi\left(NF\left(\frac{k}{M}\right)\right)\lesssim \left(M^{5/3}\mu^{1/3}\right)^{\frac{131}{416}}(\log\mu)^{\frac{18627}{8320}}.
\end{equation*}
Summing this bound over $j$ finishes the proof.



\end{document}